\documentclass[reqno]{amsart}
\usepackage{amsmath, amsthm, amsfonts, amssymb}
\usepackage{mathrsfs, color}

\allowdisplaybreaks

\def\rr{{\mathbb R}}
\def\rn{{{\rr}^n}}

\newtheorem{theorem}{Theorem}[section]
\newtheorem{corollary}{Corollary}[section]
\newtheorem{lemma}[theorem]{Lemma}

\theoremstyle{definition}
\newtheorem{definition}[theorem]{Definition}

\theoremstyle{remark}
\newtheorem{remark}[theorem]{Remark}

\numberwithin{equation}{section}

\begin{document}
\arraycolsep=1pt

\title{WEIGHTED MULTILINEAR HARDY OPERATORS AND COMMUTATORS}

\author{Zun Wei Fu}
\address{Department of Mathematics, Linyi University, Linyi Shandong, 276005, P.R. China}
\email{zwfu@mail.bnu.edu.cn}
\thanks{This work was partially supported by
 NSF of China (Grant No. 11271175, 10901076, 10931001 and 11101038).}

\author{Shu Li Gong}
\address{College of Mathematics and
Economtrics, Hunan University, Changsha 410082, P.R. China}
\email{gongshuli123@163.com}

\author{Shan Zhen Lu}
\address{School of Mathematical Sciences, Beijing Normal University,
Laboratory of Mathematics and Complex Systems, Ministry of
Education, Beijing 100875, P.R. China}
\email{lusz@bnu.edu.cn}

\author{Wen Yuan}
\address{School of Mathematical Sciences, Beijing Normal University,
Laboratory of Mathematics and Complex Systems, Ministry of
Education, Beijing 100875, P.R. China}
\email{wenyuan@bnu.edu.cn}


\date{}


\begin{abstract}

In this paper, we introduce a type of
weighted multilinear Hardy operators
and obtain their sharp bounds on the product of Lebesgue spaces and central Morrey spaces.
In addition, we  obtain  sufficient and necessary conditions of the weight functions
so that the commutators of
the weighted multilinear Hardy operators (with symbols in
central BMO space) are bounded on the product of central Morrey spaces.
These results are further used to
prove sharp estimates of some inequalities due to Riemann-Liouville and Weyl.
\end{abstract}

\maketitle


\section{Introduction}

Let $\omega:\, [0,1]\rightarrow [0,\infty)$ be a measurable function.
The {\it weighted Hardy
operator} $H_{\omega}$ is defined on all complex-valued measurable
functions $f$ on $\mathbb R^n$ as follows:
$$H_{\omega}f(x):=\int^1_{0}f(tx)\omega(t)\,dt,\hspace{3mm}x\in \mathbb{ R}^n.$$
Under certain conditions on $\omega$, Carton-Lebrun and Fosset \cite{CF} proved that
$H_{\omega}$ maps $L^{p}(\mathbb{ R}^n)$ into
itself for $1<p<\infty$. They also pointed out that the operator $H_{\omega}$
commutes with the Hilbert transform when $n=1$, and
with certain Calder\'{o}n-Zygmund singular integral
operators including the Riesz transform when $n\geq2$.
A further extension
 of the results obtained in \cite{CF} was due to Xiao in \cite{X}.


\medskip

\noindent{\bf Theorem A  \cite{X}.\,}
Let $1<p<\infty$ and $\omega :
[0,1]\rightarrow [0,\infty)$ be a measurable function. Then,
$H_{\omega}$ is bounded on $L^{p}(\mathbb{ R}^n)$ if and only if
$$\mathbb{A}:=\int^{1}_{0}t^{-n/p}\omega(t)\,dt<\infty. \eqno(1.1)$$
Moreover, $$\|H_{\omega}f\|_{L^{p}(\mathbb{R}^n)
\rightarrow L^{p}(\mathbb{R}^n)}=\mathbb{A}.\eqno(1.2)$$

\medskip

Notice that the condition (1.1) implies that $\omega$ is integrable on [0,1].
The constant $\mathbb{A}$  seems to be of interest as it equals
to $\frac{p}{p-1}$ if $\omega\equiv 1$ and $n=1$. In this case,
$H_{\omega}$ is reduced to the {\it classical Hardy operator} $H$ defined by
$$Hf(x):=\frac{1}{x}\int^x_{0}f(t)\,dt,\, x\neq0,$$
which is the most fundamental averaging operator in analysis.
Also, a celebrated integral inequality,
due to Hardy \cite{HLP}, can be deduced from Theorem A immediately
$$\|Hf\|_{L^{p}(\mathbb{R})}\leq \frac{p}{p-1}\|f\|_{L^{p}(\mathbb{R})},
\eqno(1.3)$$ where $1<p<\infty$ and the constant $\frac{p}{p-1}$ is
the best possible.

Another interesting application of Theorem A is  the sharp estimate of
the Riemann-Liouville integral operator on the Lebesgue spaces.
To be precise, let $n=1$ and we take
$$\omega(t):=\frac{1}{\Gamma(\alpha)(1-t)^{1-\alpha}}, \quad t\in[0,1],$$
where $0<\alpha<1$. Then
$$H_{\omega}f(x)=x^{-\alpha}I_{\alpha}f(x),\quad x>0,$$
where $I_{\alpha}$ is the {\it Riemann-Liouville integral operator}
defined by
$$I_{\alpha}f(x):=\frac{1}{\Gamma(\alpha)}
\int_{0}^{x}\frac{f(t)}{(x-t)^{1-\alpha}}\,dt, \quad x>0.$$
Note that the operator $I_{\alpha}$
is exactly the one-sided version of the well-known Riesz potential
$$\mathcal{I}_{\alpha}f(x):=C_{n, \alpha}\int_{\mathbb{R}^{n}}
\frac{f(t)}{|x-t|^{n-\alpha}}\,dt,~~~x\in\rr.$$
Clearly, Theorem A implies the celebrated result of
Hardy, Littlewood  and  Polya in [8, Theorem 329], namely,
for all $0<\alpha<1$ and $1<p<\infty$,
$$\|I_{\alpha}\|_{L^{p}(\rr)\to L^p(x^{-p\alpha} dx)}
=\frac{\Gamma(1-1/p)}{\Gamma(1+\alpha-1/p)}.\eqno(1.4)$$

Now we recall the commutators of weighted Hardy operators introduced
in \cite{FLL}. For any locally integrable function $b$ on $\mathbb{R}^n$
and integrable function
$\omega :\, [0,1]\rightarrow [0,\infty)$, the {\it commutator
of the weighted Hardy operator} $H_{\omega}^{b}$ is defined by
$$H_{\omega}^{b}f:=bH_{\omega}f-H_{\omega}(bf).$$

It is easy to see that the commutator
$H_{\omega}^{b}$ is bounded on $L^{p}(\mathbb{R}^{n})$ for $1<p<\infty$ when $b\in L^{\infty}({\mathbb{R}}^{n})$ and
$\omega$ satisfies the condition (1.1).
An interesting choice of $b$ is that it belongs
to the class of $\mathrm{BMO}({\mathbb{R}}^{n})$.
Recall that $\mathrm{BMO}({\mathbb{R}}^{n})$ is defined to be the space of
all $b\in L_{loc}{({\mathbb{R}}^{n})}$ such that
$$\|b\|_{BMO}:=\sup_{Q\subset\mathbb{R}^{n}}\frac{1}{|Q|}\int_{Q}|b(x)-b_{Q}|
\,dx< \infty,$$
where $b_{Q}:=\frac{1}{|Q|}\int_{Q}b$ and the supremum is taken over all
cubes $Q$ in ${\mathbb{R}^n}$ with sides parallel to the axes.
It is well known that
$L^{\infty}({\mathbb{R}}^{n})\varsubsetneq \mathrm{BMO}({\mathbb{R}}^{n})$
since $\mathrm{BMO}({\mathbb{R}}^{n})$ contains unbounded functions such as $\log|x|$.

When symbols $b\in\mathrm{BMO}({\mathbb{R}}^{n})$,
the condition (1.1) on weight functions $\omega$ does not ensure
the boundedness of $H_{\omega}^{b}$ on $L^{p}(\mathbb{R}^{n})$.
Via controlling $H_{\omega}^{b}$ by the Hardy-Littlewood maximal operators instead of
sharp maximal functions, Fu, Liu and Lu \cite{FLL} established sufficient and necessary
conditions on weight functions $\omega$
which ensure that $H_{\omega}^{b}$ is bounded on
$L^{p}(\mathbb{R}^{n})$ when $1<p<\infty$. Precisely,
they obtain the following conclusion.
\medskip

\noindent{\bf Theorem B.\,}
Let
\[
\mathbb{C}:=\int^{1}_{0}t^{-n/p}\omega(t)\log\frac{2}{t}\,dt
\]
and $1<p<\infty$. Then following statements are equivalent:\\
$\rm(i)$\quad $\omega$ is integrable and
$H^{b}_\omega$ is  bounded  on $L^{p}(\mathbb{R}^{n})$
for all $b\in
\mathrm{BMO}(\mathbb{R}^{n})$;\\
$\rm(ii)$\quad $\mathbb{C}<\infty.
$
\medskip

We remark that the condition (1.1), i.\,e., $\mathbb{A}<\infty$, is
weaker than $\mathbb{C}<\infty$ in Theorem B.
In fact, if we let
\[
\mathbb{B}:=\int^{1}_{0}t^{-n/p}\omega(t)\log\frac{1}{t}\,dt,
\]
then
 $\mathbb{C}=\mathbb{A}\log2+\mathbb{B}$. Hence
$\mathbb{C}<\infty$ implies $\mathbb{A}<\infty$. However, $\mathbb{A}<\infty$
can not imply $\mathbb{C}<\infty$. To see this, for $0<\alpha<1$,
let
\[
e^{s(n/p-1)}\tilde{\omega}(s)=\left\{
\begin{array}{ll}
s^{-1+\alpha},&\quad 0<s\leq 1,\\
s^{-1-\alpha},&\quad 1<s<\infty,\\
0,&\quad s=0, \infty \end{array}\right.\eqno(1.5)
\]
and $\omega(t):=\tilde{\omega}(\log\frac{1}{t})$, $0\leq t\leq1$. Then
it is not difficult to verify $\mathbb{A}<\infty$ and $\mathbb{C}=\infty$.

Later on in \cite{FZW}, the conclusions in Theorems A and B were
further generalized  to the central Morrey spaces $\dot{B}^{p,\lambda}(\rn)$ and the central
BMO space $C\dot{M}O^q(\rn)$. Here the space $C\dot{M}O^q(\rn)$ was first
introduced by Lu and Yang in \cite{LY2}, and
the space $\dot{B}^{p,\lambda}(\rn)$ is a generalization of
$C\dot{M}O^q(\rn)$
introduced by Alvarez, Guzman-Partida and Lakey in
\cite{AGL}; see also \cite{bg}.

\begin{definition} Let $\lambda\in
\mathbb{R}$ and $1<p<\infty$. The \emph{central Morrey space}
$\dot{B}^{p,\,\lambda}(\mathbb{R}^{n})$ is defined to
be the space of all locally $p$-integrable functions $f$ satisfying that
$$\|f\|_{\dot{B}^{p,\,\lambda}}=\sup_{R>0}\biggl(\frac{1}{|B(0, R)|^{1+\lambda p}}
\int_{B(0, R)}|f(x)|^{p}dx\biggr)^{1/p}<\infty.$$
\end{definition}

Obviously, $\dot{B}^{p,\,\lambda}(\rn)$ is a Banach space. One can easily
check that $\dot{B}^{p,\lambda}(\mathbb{R}^n)=\{0\}$ if $\lambda<-1/p$, $\dot{B}^{p,0}(\mathbb{R}^n)=\dot{B}^{p}(\mathbb{R}^n)$,
$\dot{B}^{q,-1/q}(\mathbb{R}^n)=L^{q}(\mathbb{R}^n)$, and
$\dot{B}^{p,\lambda}(\mathbb{R}^n)\supsetneq L^{p}(\mathbb{R}^n)$
if $\lambda>-1/p$, where the space $\dot{B}^{p}(\mathbb{R}^n)$ was
introduced by Beurling in \cite{B}.
Similar to the classical Morrey space,
we only consider the case $-1/p<\lambda\leq0$ in this paper.
In the past few years, there is an increasing interest
on the study of Morrey-type spaces and their various generalizations
and the related theory of operators; see, for example,
\cite{AGL,GAK,FZW,MN,KMNY}.

\begin{definition} Let
$1<q<\infty$. A function $f\in L_{\mathrm{loc}}^{q}(\mathbb{R}^{n})$ is said
to belong to the \emph{central bounded mean oscillation space}
$C\dot{M}O^{q}(\mathbb{R}^{n})$ if
$$\|f\|_{C\dot{M}O^{q}}=\sup_{R>0}\biggl(\frac{1}{|B(0, R)|}
\int_{B(0, R)}|f(x)-f_{B(0,\, R)}|^{q}dx\biggr)^{1/q}<\infty.\eqno(1.6)$$
\end{definition}

The space $C\dot{M}O^{q}(\rn)$ is a Banach
space in the sense that two functions that differ by
a constant are regarded as a function in this space.
Moreover, {\rm(1.6)} is equivalent
to the following condition
$$\sup_{R>0}\inf_{c\in\mathbb{C}}\biggl(\dfrac{1}{|B(0, R)|}
\int_{B(0, R)}|f(x)-c|^{q}dx\biggr)^{1/q}<\infty.$$
For more detailed properties of these two spaces, we refer to \cite{FZW}.

For $1<p<\infty$ and $-1/p< \lambda\le0$, it was proved in \cite[Theorem 2.1]{FZW} that
the  norm
$$\|H_w\|_{\dot{B}^{p,\,\lambda}(\rn)\to \dot{B}^{p,\,\lambda}(\rn)}
=\int_0^1 t^{n\lambda} w(t)\,dt.$$
Moreover, if $1<p_1<p<\infty$, $1/p_1=1/p+1/q$ and $-1/p<\lambda<0$, then it was proved
in \cite[Theorem 3.1]{FZW} that $H^b_w$ is bounded from $\dot{B}^{p,\,\lambda}(\rn)$
to $\dot{B}^{p_1,\,\lambda}(\rn)$ if and only if
$$\int_0^1 t^{n\lambda} w(t) \log\frac2{t}\,dt<\infty,$$
where the symbol $b\in C\dot{M}O^{q}(\rn)$.

In this paper, we consider the multilinear version of the above results.
Recall that the weighted multilinear Hardy operator is defined as follows.

\begin{definition}
Let $m\in\mathbb{N}$, and
$$\omega:\, \overbrace{{[0,1]\times[0,1]\times\cdots\times[0,1]}}^{m\,\text{times}}
\rightarrow [0,\infty)$$
be an integrable function.
The {\it weighted multilinear Hardy operator $\mathcal{H}_{\omega}^m$}
is defined as
\[
\mathcal{H}_{\omega}^m(\vec{f})(x):=
\int\limits_{0<t_{1},t_{2},\ldots,t_{m}<1}
\left(\prod_{i=1}^{m}f_{i}(t_{i}x)\right)\omega(\vec{t})\,d\vec t,\quad x\in \mathbb{R}^n,
\]
where $\vec f:=(f_1,\ldots, f_m)$, $\omega(\vec{t}):=\omega(t_{1},t_{2},\ldots,t_{m})$, $d\vec t:= dt_1\,\cdots\,dt_m$,
 and $f_{i}~(i=1,\ldots,m)$ are
 complex-valued measurable functions on $\mathbb{R}^n$.
When $m=2$, $\mathcal{H}_{\omega}^m$ is referred to as bilinear.
\end{definition}

The study of multilinear averaging operators is traced to the
multilinear singular integral operator theory (see, for example, \cite{CM}), and motivated
not only the generalization of the theory of linear ones but also their
natural appearance in analysis. For a more complete account on multilinear operators, we
refer to \cite{FL}, \cite{GL}, \cite{L} and the references
therein.

The main aim of the paper is to establish the sharp bounds of
weighted multilinear Hardy operators on the product
of Lebesgue spaces and central Morrey spaces. In addition,
we find sufficient and necessary conditions of the weight functions
so that commutators of such weighted multilinear Hardy operators (with symbols in
$\lambda$-central BMO space) are bounded on the product of central Morrey spaces.

The paper is organized as follows: Section 2 is devoted to the sharp estimates of $\mathcal{H}_{\omega}^m$ on the products of Lebesgue spaces and also
central Morrey spaces. In Section 3, we  present the sharp estimates of the commutator generated by $\mathcal{H}_{\omega}^m$ with symbols in $\dot{CMO}^q(\rn)$.
Section 4 focuses on weighted Ces\`{a}ro operators
of multilinear type related to weighted multilinear Hardy operators.

\section{Sharp boundedness of $\mathcal{H}_{\omega}^m$ on the
product of central Morrey spaces}

We begin with the following sharp boundedness of $\mathcal{H}_{\omega}^m$ on the product of
Lebesgue spaces, which when $m=1$ goes back to Theorem A.

\begin{theorem}\label{t1}
Let  $1<p, p_i<\infty$, $i=1,\ldots, m$ and  $1/p=1/p_1+\cdots+1/p_m$.
Then, $\mathcal{H}_{\omega}^m$ is bounded from
$L^{p_1}(\rn)\times \dots \times L^{p_m}(\rn)$ to $ L^p(\rn)$ if and only
if
\begin{eqnarray}\label{A}
\mathbb{A}_m:=\int\limits_{0<t_{1},t_{2},...,t_{m}<1}
\left(\prod_{i=1}^{m}t_{i}^{-n/p_{i}}\right)\omega(\vec{t})\,d\vec{t}<\infty.
\end{eqnarray}
Moreover,
$$\|\mathcal{H}_{\omega}^m\|_{L^{p_1}(\rn)\times \dots \times L^{p_m}(\rn)
\rightarrow L^{p}(\rn)}=\mathbb{A}_m.$$
\end{theorem}

\begin{proof}[Proof]

 In order to simplify the proof, we only
consider the case that $m=2$. Actually,  a similar procedure
works for all $m\in \mathbb{N}$.

Suppose that (\ref{A}) holds. Using Minkowski's inequality yields
$$\begin{array}{rl}
\displaystyle \|\mathcal{H}_{\omega}^2(f_{1}, f_{2})
\|_{L^p(\mathbb{R}^{n})}
&=\displaystyle\left(
\int_{\mathbb{R}^n}\left|\int\limits_{0<t_{1},t_{2}<1}f_{1}
(t_{1}x)f_{2}(t_{2}x)\omega(t_{1},t_{2})\,dt_{1}dt_{2}
\right|^{p}dx\right)^{1/p}\\
&\leq\displaystyle \int\limits_{0<t_{1},t_{2}<1}\left(\int_{\mathbb{R}^n}\left|f_{1}(t_{1}x)f_{2}(t_{2}x)\right|^{p}dx\right)^{1/p}\omega(t_{1},t_{2})\,dt_{1}dt_{2}.
\end{array}$$
By H\"{o}lder's inequality with $1/p=1/p_{1}+1/p_{2}$, we see that
$$\begin{array}{rl}
\displaystyle \|\mathcal{H}_{\omega}^2(f_{1}, f_{2})\|_{L^p(\mathbb{R}^{n})}&\leq\displaystyle \int\limits_{0<t_{1},t_{2}<1}\prod_{i=1}^{2}\left(\int_{\mathbb{R}^n}\left|f_{i}(t_{i}x)\right|^{p_{i}}dx\right)^{1/p_{i}}\omega(t_{1},t_{2})\,dt_{1}dt_{2}\\
&\leq\displaystyle \left(\prod_{i=1}^{2}\|f_{i}\|_{L^{p_i}(\mathbb{R}^{n})}\right)\int\limits_{0<t_{1},t_{2}<1}\left(\prod_{i=1}^{2}t_{i}^{-n/p_{i}}\right)\omega(t_{1},t_{2})\,dt_{1}dt_{2}.\end{array}$$
Thus, $\mathcal{H}_{\omega}^2$ maps
the product of Lebesgue spaces
$L^{p_1}(\mathbb{R}^{n})\times L^{p_2}(\mathbb{R}^{n})$ to $
L^p(\mathbb{R}^{n})$
and
\begin{eqnarray}\label{2.1}
\|\mathcal{H}_{\omega}^2\|_{L^{p_1}(\mathbb{R}^{n})\times L^{p_2}(\mathbb{R}^{n})\rightarrow
L^p(\mathbb{R}^{n})}\leq\mathbb{A}_2.
\end{eqnarray}

To see the necessity, for sufficiently small $\varepsilon\in (0, 1)$, we set
\begin{eqnarray}\label{2.2}
f^{\varepsilon}_{1}(x):=
\begin{cases}
0,&\quad |x|\leq\frac{\sqrt{2}}{2},\\
\displaystyle|x|^{-\frac{n}{p_1}-\frac{p_2\varepsilon}{p_1}},&\quad
|x|>\frac{\sqrt{2}}{2},\end{cases}
\end{eqnarray}
and
\begin{eqnarray}\label{2.3}
f^{\varepsilon}_{2}(x):=
\begin{cases}
0,&\quad |x|\leq\frac{\sqrt{2}}{2},\\
\displaystyle|x|^{-\frac{n}{p_2}-\varepsilon},&\quad
|x|>\frac{\sqrt{2}}{2}.
\end{cases}
\end{eqnarray}
An elementary calculation gives  that
$$
\|f_1^\varepsilon\|_{L^{p_1}(\mathbb{R}^{n})}^{p_1}
=\|f_2^\varepsilon\|_{L^{p_2}(\mathbb{R}^{n})}
^{p_2}=\frac{\omega_n}{p_2\varepsilon}
\Big(\frac{\sqrt{2}}{2}\, \Big)^{-p_2\varepsilon},
$$
where $\omega_n=\frac{n\pi^{n/2}}{\Gamma(1+n/2)}$ is the volume of the unit sphere.
Consequently, we have
\begin{eqnarray*}
&&\|\mathcal{H}_{\omega}^2(f_{1}^\varepsilon, f_{2}^\varepsilon)\|_{L^p(\mathbb{R}^{n})}\\
&&\hspace{0.2cm}
=\left\{\int_{\mathbb{R}^n}
|x|^ {-n-p_2\varepsilon}
\left[\int\limits_{E_{x}(t_{1}, t_{2})}t_{1}^{-\frac{n}{p_1}
-\frac{p_2\varepsilon}{p_1}}
t_{2}^{-\frac{n}{p_2}-\varepsilon}\omega(t_{1},t_{2})
\,dt_{1}dt_{2}\right]^p\,dx\right\}^{1/p},
\end{eqnarray*}
where
\[
E_{x}(t_{1}, t_{2}):=\left\{(t_{1}, t_{2})|\, 0<t_{1},t_{2}<1;\,
t_1>\frac{\sqrt {2}}{2|x|};\,
t_2>\frac{\sqrt {2}}{2|x|}\right\}.
\]
Hence,
$$\begin{array}{rl}\displaystyle&\|\mathcal{H}_{\omega}^2
(f^{\varepsilon}_{1}, f^{\varepsilon}_{2})(x)\|^{p}_{L^p(\mathbb{R}^{n})}\\
&\hspace{0.2cm}\displaystyle\geq\int_{|x|>1/\varepsilon}|x|^{-n-p_2\varepsilon}
\left(\int\limits_{E_{\frac{1}{\varepsilon}}(t_{1}, t_{2})}
t_{1}^{-\frac{n}{p_1}-\frac{p_2\varepsilon}{p_1}}
t_{2}^{-\frac{n}{p_2}-\varepsilon}\omega(t_{1},t_{2})dt_{1}dt_{2}\right)^{p}dx
 \\
 \displaystyle
&\hspace{0.2cm}=\displaystyle\frac{\varepsilon^{p_2\varepsilon}\omega_{n}}{p_2 \varepsilon}
\left(\int\limits_{E_{\frac{1}{\varepsilon}}(t_{1}, t_{2})} t_{1}^{-\frac{n}{p_1}-\frac{p_2\varepsilon}{p_1}}
t_{2}^{-\frac{n}{p_2}-\varepsilon}\omega(t_{1},t_{2})dt_{1}dt_{2}\right)^{p}
\\
\displaystyle
&\hspace{0.2cm}=\displaystyle\left(\frac{\sqrt{2}}{2}\varepsilon
\right)^{p_2\varepsilon}\prod_{i=1}^{2}\|f_{i}^\varepsilon
\|_{L^{p_i}(\mathbb{R}^{n})}^{p}\displaystyle\left(\int\limits_{E_{\frac{1}{\varepsilon}}(t_{1}, t_{2})}t_{1}^{-\frac{n}{p_1}-\frac{p_2\varepsilon}{p_1}}
t_{2}^{-\frac{n}{p_2}-\varepsilon}\omega(t_{1},t_{2})dt_{1}dt_{2}\right)^{p}.
\end{array}$$
Therefore,
\begin{eqnarray*}
&&\|\mathcal{H}_{\omega}^2\|_{L^{p_1}(\mathbb{R}^{n})\times L^{p_2}(\mathbb{R}^{n})
\rightarrow
L^p(\mathbb{R}^{n})}\\
&&\hspace{0.2cm}\geq\displaystyle \left(\frac{\sqrt{2}}{2}
\varepsilon\right)^{p_2\varepsilon/p}
\int\limits_{E_{\frac{1}{\varepsilon}}(t_{1}, t_{2})}
t_{1}^{-\frac{n}{p_1}-\frac{p_2\varepsilon}{p_1}}
t_{2}^{-\frac{n}{p_2}-\varepsilon}\omega(t_{1},t_{2})\,dt_{1}\,dt_{2}.
\end{eqnarray*}
Since $(\sqrt{2}\varepsilon/2)^{p_2\varepsilon/p}\rightarrow1$
as  $\varepsilon\rightarrow 0^{+}$, by letting $\varepsilon\rightarrow 0^{+}$, we know that
\begin{eqnarray}\label{2.4}
\|\mathcal{H}_{\omega}^2\|_{L^{p_1}(\mathbb{R}^{n})\times
L^{p_2}(\mathbb{R}^{n})\rightarrow
L^p(\mathbb{R}^{n})}\geq \mathbb{A}_2.
\end{eqnarray}
Combining (\ref{2.1}) and (\ref{2.4}) then finishes the proof.
\end{proof}

Observe that when $n=1$ and $\alpha\in(0,m)$, if we take
$$\omega(\vec{t}):=\frac{1}{\Gamma(\alpha)|(1-t_{1}, \dots, 1-t_{m})|^{m-\alpha}},$$
then
$$\mathcal{H}_{\omega}^m(\vec{f})(x)=x^{-\alpha}I^{m}_{\alpha}\vec{f}(x),\quad x>0,$$
where
$$I^{m}_{\alpha}\vec{f}(x):=\frac{1}{\Gamma(\alpha)}
\int\limits_{0<t_{1},t_{2},...,t_{m}<x}
\frac{\prod_{i=1}^{m}f_{i}(t_{i})}{|(x-t_{1}, \dots, x-t_{m})|^{m-\alpha}}\,d\vec{t}.$$
The operator $I^{m}_{\alpha}$ turns out to be
the one-sided analogous to the one-dimensional multilinear Riesz operator
$\mathcal{I}^{m}_{\alpha}$ studied by
Kenig and Stein in [9], where
$$\mathcal{I}^{m}_{\alpha}\vec{f}(x):=
\int\limits_{t_{1},t_{2},...,t_{m}\in\mathbb{R}}
\frac{\prod_{i=1}^{m}f_{i}(t_{i})}{|(x-t_{1}, \dots, x-t_{m})|^{m-\alpha}}\,d\vec{t},\qquad x\in\rr.$$
As an application of Theorem \ref{t1} we obtain the following
sharp estimate of the boundedness of $I^{m}_{\alpha}$.

\begin{corollary}
Let $0<\alpha<m$. With the same assumptions as in Theorem \ref{t1},
the operator $I^{m}_{\alpha}$ maps $L^{p_1}(\rr)\times \dots \times L^{p_m}(\rr)$ to $ L^p(x^{-p\alpha} dx)$
and the operator norm equals to
$$\frac{1}{\Gamma(\alpha)}\int\limits_{0<t_{1},t_{2},\ldots,t_{m}<1}\left(\prod_{i=1}^{m}
t_{i}^{-1/p_{i}}\right)\frac{1}{|(1-t_{1}, \dots, 1-t_{m})|^{m-\alpha}}\,d\vec{t}.$$
\end{corollary}

Next we extend the result in Theorem \ref{t1} to the product of central Morrey spaces.

\begin{theorem}\label{t2}
Let $1<p<p_i<\infty,$ $1/p=1/p_1+\cdots+1/p_m$, $\lambda=\lambda_1+\cdots+\lambda_m$ and
$-1/p_i\leq \lambda_i<0~(i=1,2,\ldots,m)$.

{\rm (i)} If
\begin{eqnarray}\label{Am}
\widetilde{\mathbb{A}}_{m}:
=\int\limits_{0<t_1,t_2,\ldots,t_m<1}\left(\prod_{i=1}^m t_i^{n\lambda_i}\right)\omega(\vec{t})d\vec{t}<\infty,
\end{eqnarray}
then $\mathcal{H}_\omega^m$ is bounded from $\dot{B}^{p_1,\lambda_1}(\rn)\times\cdots
\times\dot{B}^{p_m,\lambda_m}(\rn)$ to $\dot{B}^{p,\lambda}(\rn)$
with its operator norm not more that $\widetilde{\mathbb{A}}_{m}$.

{\rm (ii)} Assume that $\lambda_1p_1=\cdots=\lambda_mp_m$. In this case the condition
\eqref{Am} is also  necessary  for the boundedness of
$\mathcal{H}_\omega^m:\ \dot{B}^{p_1,\lambda_1}(\rn)\times\cdots
\times\dot{B}^{p_m,\lambda_m}(\rn)\to\dot{B}^{p,\lambda}(\rn)$.
Moreover, $$\|\mathcal{H}_\omega^m\|_{\dot{B}^{p_1,\lambda_1}(\rn)
\times\cdots\times\dot{B}^{p_m,\lambda_m}(\rn)
\rightarrow\dot{B}^{p,\lambda}(\rn)}=\widetilde{\mathbb{A}}_{m}.$$
\end{theorem}

\begin{proof}
By similarity, we only give the proof in the case $m=2$.

When $-1/p_i=\lambda_i$, $i=1,2$, then Theorem \ref{t2} is just Theorem \ref{t1}.

Next we consider the case that $-1/p_i<\lambda_i<0$, $i=1,2$.

First, we assume $\widetilde{\mathbb{A}}_2<\infty$.
Since $1/p=1/p_1+1/p_2$, by Minkowski's inequality and H\"{o}lder's inequality, we see that,
for all balls $B=B(0,R)$,
\begin{eqnarray}\label{H2f}
&&\left(\frac{1}{|B|^{1+\lambda p}}\int_B|\mathcal{H}_{\omega}^2(\vec{f})(x)|^pdx\right)^{1/p}\nonumber\\
& &\hspace{0.2cm}\leq \int\limits_{0<t_1,t_2<1}\left(\frac{1}{|B|^{1+\lambda p}}\int_B\Big|\prod_{i=1}^2f_i(t_i x)\Big|^pdx\right)^{1/p}\omega(\vec{t})d\vec{t}\nonumber\\
& &\hspace{0.2cm}\leq \int\limits_{0<t_1,t_2<1}\prod_{i=1}^2\left(\frac{1}{|B|^{1+\lambda_i p_i}}\int_B\Big|f_i(t_i x)\Big|^{p_i}dx\right)^{1/p_i}\omega(\vec{t})d\vec{t}\nonumber\\
&&\hspace{0.2cm}= \int\limits_{0<t_1,t_2<1}t_1^{n\lambda_1}t_2^{n\lambda_2}\prod_{i=1}^2\left(\frac{1}{|t_i B|^{1+\lambda_i p_i}}\int_{t_i B}\Big|f_i( x)\Big|^{p_i}dx\right)^{1/p_i}\omega(\vec{t})d\vec{t}\nonumber\\
&&\hspace{0.2cm}\le \|f_1\|_{\dot{B}^{p_1,\lambda_1}}\|f_2\|_{\dot{B}^{p_2,\lambda_2}}
\int\limits_{0<t_1,t_2<1}t_1^{n\lambda_1}t_2^{n\lambda_2}\omega(\vec{t})d\vec{t}.
\end{eqnarray}
This means that $\|\mathcal{H}_\omega^2\|_{\dot{B}^{p_1,\lambda_1}(\rn)\times\dot{B}^{p_2,\lambda_2}(\rn)
\rightarrow\dot{B}^{p,\lambda}(\rn)}\le\widetilde{\mathbb{A}}_{2}.$

For the necessity when $\lambda_1p_1=\lambda_2p_2$, let $f_1(x):=|x|^{n\lambda_1}$ and $f_2(x):=|x|^{n\lambda_2}$ for all $x\in\rn\setminus\{0\}$, and $f_1(0)=f_2(0):=0$.
Then for any $B:=B(0,R)$,
\begin{eqnarray*}
\left(\frac{1}{|B|^{1+\lambda_i p_i}}\int_B|f_i(x)|^{p_i}dx\right)^{1/p_i}&=&\left(\frac{1}{|B|^{1+\lambda_i p_i}}\int_B|x|^{n\lambda_i p_i}dx\right)^{1/p_i}=
\left(\frac{\omega_n}{n}\right)^{-\lambda_i}\left(\frac{1}{1+\lambda_ip_i}\right)^{1/p_i}.
\end{eqnarray*}
Hence  $\|f_i\|_{\dot{B}^{p_i,\lambda_i}}
=(\omega_n/n)^{-\lambda_i}(\frac{1}{n+n\lambda_ip_i})^{1/p_i}$,
$i=1,2$. Since $\lambda=\lambda_1+\lambda_2$ and $-1/p_i< \lambda_i<0, 1<p<p_i<\infty,~i=1,2$, we have
\begin{eqnarray}
&&\left(\frac{1}{|B|^{1+\lambda p}}\int_B|\mathcal{H}_{\omega}^2(\vec{f})(x)|^{p}dx\right)^{1/p} \nonumber\\
&&\hspace{0.2cm}=\left(\frac{1}{|B|^{1+\lambda p}}\int_B|x|^{n\lambda p}\,dx\right)^{1/p}\int\limits_{0<t_1,t_2<1}t_1^{n\lambda_1}
t_2^{n\lambda_2}\omega(\vec{t})d\vec{t} \nonumber\\
&&\hspace{0.2cm}=\left(\frac{\omega_n}{n}\right)^{-\lambda}\left(\frac{1}{1+\lambda p}\right)^{1/p}\int\limits_{0<t_1,t_2<1}t_1^{n\lambda_1}
t_2^{n\lambda_2}\omega(\vec{t})d\vec{t} \nonumber\\
&&\hspace{0.2cm}= \|f_1\|_{\dot{B}^{p_1,\lambda_1}}
\|f_2\|_{\dot{B}^{p_2,\lambda_2}} \frac{(1+\lambda_1p_1)^{1/p_1}
(1+\lambda_2p_2)^{1/p_2}}{(1+\lambda p)^{1/p}}
\int\limits_{0<t_1,t_2<1}t_1^{n\lambda_1}t_2^{n\lambda_2}\omega(\vec{t})d\vec{t}. \nonumber\\
&&\hspace{0.2cm}= \|f_1\|_{\dot{B}^{p_1,\lambda_1}}
\|f_2\|_{\dot{B}^{p_2,\lambda_2}}
\int\limits_{0<t_1,t_2<1}t_1^{n\lambda_1}t_2^{n\lambda_2}\omega(\vec{t})d\vec{t}, \label{eq2-6}
\end{eqnarray}
since $\lambda_1p_1=\lambda_2p_2.$
Then,
$\widetilde{\mathbb{A}}_2\leq \|\mathcal{H}_{\omega}^2\|_{\dot{B}^{p_1,\lambda_1}\times
\dot{B}^{p_2,\lambda_2}\rightarrow\dot{B}^{p,\lambda}}<\infty.$
Combining  (\ref{H2f}) and (\ref{eq2-6}) then concludes the proof.
This finishes the proof of  the Theorem \ref{t2}.
\end{proof}

We remark that Theorem \ref{t2} when $m=1$ goes back to \cite[Theorem 2.1]{FZW}.

A corresponding conclusion of $I^\alpha_m$ is also true.

\begin{corollary}
Let $0<\alpha<m$. With the same assumptions as in Theorem \ref{t2},
the operator $I^{m}_{\alpha}$ maps $\dot{B}^{p_1,\lambda_1}(\rr)\times \cdots\times\dot{B}^{p_m,\lambda_m}(\rr)$ to $\dot{B}^{p,\lambda}(x^{-p\alpha}dx)$ with
the operator norm not more than
$$\frac{1}{\Gamma(\alpha)}\int\limits_{0<t_{1},t_{2},\ldots,t_{m}<1}\left(\prod_{i=1}^{m}
t_{i}^{\lambda_{i}}\right)\frac{1}{|(1-t_{1}, \dots, 1-t_{m})|^{m-\alpha}}\,d\vec{t}.$$
In particular, when  $\lambda_1p_1=\cdots=\lambda_mp_m$, then
the operator norm of $I^{m}_{\alpha}$ equals to the above quantity.
\end{corollary}

\begin{remark}
Notice that in the necessary part of Theorem \ref{t2},
we need an additional condition $\lambda_1p_1=\cdots=\lambda_mp_m$.
In the case of Lebesgue spaces, this condition
holds true automatically. For the case of
Morrey spaces, such
condition has known to be the necessary and sufficient condition for the
interpolation properties of Morrey spaces; see, for example, \cite{LR}.
\end{remark}


\section{Commutators of weighted multilinear Hardy operators}

In this section, we consider the sharp estimates of the multilinear
commutator generated by $\mathcal{H}_{\omega}^m$ with symbols in $\dot{CMO}^q(\rn)$.
Before presenting the main results of this section, we first introduce the following well-known Riemann-Lebesgue-type Lemma, which plays a key role in the below proof.
For completeness, we give a detailed proof.

\begin{lemma}\label{LA}
Let $m\in\mathbb{N}$ and $\omega:\,[a,b]^m\to[0,\infty)$ be an
integrable function. Then
$$\lim_{r\to\infty}\int_{[a,b]^m}\omega(t_1,\cdots,t_m)\,\prod_{i\in E}
\sin(\pi r t_i)\,dt_1\,\cdots\,dt_m=0,$$ where $E$ is an arbitrary
nonempty subset of $\{1,\cdots,m\}$.
\end{lemma}

\begin{proof}
For simplicity, we only give the proof for the case that $m=2$ and
$E=\{1\}$, namely, to show
$$\lim_{r\to\infty}\int_{[a,b]^2}\omega(t_1,t_2)\,
\sin(\pi r t_1)\,dt_1\,dt_2=0.$$

Since $\omega$ is integrable, for any $\varepsilon>0$, there exists a
partition
$\{I_i\times
J_j:\
i=1,\cdots,k\hspace{0.3cm}\mathrm{and}\hspace{0.3cm}j=1,\cdots,l\} $
such that $I_i=[a_{I_i},b_{I_i}]$, $J_j=[a_{J_j},b_{J_j}]$, $[a,b]=\cup_{i=1}^k I_i=\cup_{j=1}^l J_j$, $I_i\cap I_j=\emptyset=J_i\cap J_j$
if $i\neq j$, and
$$0\le \int_a^b\int_a^b \omega(t_1,t_2)\,dt_1\,dt_2-\sum_{i=1}^k
\sum_{j=1}^m m_{ij}|I_i||J_j|<\varepsilon/2,$$
where $m_{ij}$ is the minimum value of $\omega$ on $I_i\times J_j$.
Let
$$g(t_1,t_2):= \sum_{i=1}^k
\sum_{j=1}^m m_{ij}\chi_{I_i}(t_1)\chi_{J_j}(t_2),\quad t_1,t_2\in[a,b].$$
Then
 $$\int_a^b\int_a^b g(t_1,t_2)\,dt_1\,dt_2=\sum_{i=1}^k
\sum_{j=1}^m m_{ij}|I_i||J_j|$$
and
$$0\le \int_a^b\int_a^b [\omega(t_1,t_2)-g(t_1,t_2)]\,dt_1\,dt_2<\varepsilon/2.$$
It follows from $\omega-g\ge 0$ that
\begin{eqnarray*}
&&\left|\int_{[a,b]^2}\omega(t_1,t_2)\,
\sin(\pi r t_1)\,dt_1\,dt_2\right|\\
&&\hspace{0.2cm}\le \left|\int_{[a,b]^2}[\omega(t_1,t_2)-g(t_1,t_2)]\,
\sin(\pi  r t_1)\,dt_1\,dt_2\right|+\left|\int_{[a,b]^2}g(t_1,t_2)\,
\sin(\pi r t_1)\,dt_1\,dt_2\right|\\
&&\hspace{0.2cm}\le \int_{[a,b]^2}[\omega(t_1,t_2)-g(t_1,t_2)]\,dt_1\,dt_2+\left|\int_{[a,b]^2}g(t_1,t_2)\,
\sin(\pi r t_1)\,dt_1\,dt_2\right|\\
&&\hspace{0.2cm}\le \varepsilon/2+\left|\frac{1}{\pi r}\sum_{i=1}^k
\sum_{j=1}^m m_{ij}|J_j|[\cos(\pi ra_{I_i})-\cos(\pi rb_{I_i})]\right|.
\end{eqnarray*}
Choosing $r$ large enough such that
$$\left|\frac{1}{\pi r}\sum_{i=1}^k
\sum_{j=1}^m m_{ij}|J_j|[\cos(\pi ra_{I_i})-\cos(\pi rb_{I_i})]\right|<\varepsilon/2,$$
we then know that
$$\left|\int_{[a,b]^2}\omega(t_1,t_2)\,
\sin(\pi r t_1)\,dt_1\,dt_2\right|<\varepsilon.$$
This finishes the proof.
\end{proof}

Now we recall the definition for the multilinear
version of the commutator of the weighted Hardy operators.
Let $m\geq 2$, $\omega :\,
[0,1]\times[0,1]^m\rightarrow [0,\infty)$
be an integrable function, and $b_{i}\ (1\leq i\leq m)$
be locally integrable functions on $\rn$. We define
$$\mathcal{H}_{\omega}^{\vec{b}}(\vec{f})(x):=\int\limits_{0<t_{1},t_{2},...,t_{m}<1}
\left(\prod_{i=1}^{m}f_{i}(t_{i}x)\right)
\left(\prod_{i=1}^{m}(b_{i}(x)-b_{i}(t_{i}x))\right)\omega(\vec{t})\,d\vec{t},\quad x\in \mathbb{ R}^n.$$

In what follows, we set
$$\mathbb{B}_{m}:=\int\limits_{0<t_{1},t_{2}<,...,<t_{m}<1}
\left(\prod_{i=1}^{m}t_{i}^{n\lambda_i}\right)
\omega(\vec{t})\prod_{i=1}^{m}\log\frac{1}{t_{i}}\,d\vec{t}$$
and
$$\mathbb{C}_{m}:=\int\limits_{0<t_{1},t_{2}<,...,<t_{m}<1}
\left(\prod_{i=1}^{m}t_{i}^{n\lambda_i}\right)
\omega(\vec{t})\prod_{i=1}^{m}\log\frac{2}{t_{i}}\,d\vec{t}.$$
Then we have the following multilinear generalization of Theorem B.

\begin{theorem}\label{t3}
Let $1<p<p_i<\infty, 1<q_i<\infty$, $-1/p_i<\lambda_i<0$, $i=1,\ldots, m$, such that $1/p=1/p_1+\cdots+1/p_m+1/q_1+\cdots+1/q_m$, $\lambda=\lambda_1+\cdots+\lambda_m $. Assume further that $\omega$ is a non-negative
integrable function on $[0,1]\times\cdots \times [0,1]$.

{\rm (i)} If $\mathbb{C}_{m}<\infty,$ then $\mathcal{H}_{\omega}^{\vec{b}} $ is bounded from
$\dot{B}^{p_1,\lambda_1}(\mathbb{R}^{n})\times \cdots \times \dot{B}^{p_m,\lambda_m}(\mathbb{R}^{n})$
to $ \dot{B}^{p,\lambda}(\mathbb{R}^{n})$
for all $\vec{b}=(b_1,b_2,\ldots,b_m)\in \dot{\mathrm{CMO}}^{q_1}(\mathbb{R}^{n})
\times\cdots\times\dot{\mathrm{CMO}}^{q_m}(\mathbb{R}^{n})$.

{\rm (ii)} Assume that $\lambda_1p_1=\cdots=\lambda_mp_m$. In this case the condition
$\mathbb{C}_{m}<\infty$ in (i) is also necessary.
\end{theorem}

\begin{remark}  It is easy to verify that
condition $\rm(ii)$ in Theorem \ref{t3}
is weaker than the condition (\ref{Am}) in Theorem \ref{t2}.
\end{remark}

\begin{proof}[Proof]
By similarity, we only consider the case that $m=2$.

We first show (i). That is, we assume  $\mathbb{C}_{2}<\infty$ and show that
$$\|\mathcal{H}_{\omega}^{\vec{b}}\|_{
\dot{B}^{p_1,\lambda_1}(\mathbb{R}^{n})\times\dot{B}^{p_2,\lambda_2}(\mathbb{R}^{n})
\rightarrow \dot{B}^{p,\lambda}(\mathbb{R}^{n})}<\infty$$
whenever $\vec b=(b_1, b_2)\in\dot{\mathrm {CMO}}^{q_1}(\rn)\times\dot{\mathrm {CMO}}^{q_2}(\rn)$.
By Minkowski's inequality we have
\begin{eqnarray*}
&& \Big(\frac{1}{|B|}\int_B |\mathcal{H}_{\omega}^{\vec{b}}(\vec{f})(x)|^p\Big)^{1/p}\\
&&\displaystyle\leq \Big(\frac{1}{|B|}\int_B\Big(\int_0^1\int_0^1\prod_{i=1}^{2}|f_i(t_i x)|\prod_{i=1}^{2}|b_i(x)-b_i(t_ix)|\omega(t_1,t_2)dt_1dt_2\Big)^pdx\Big)^{1/p}\\
&&\leq \int_0^1\int_0^1\Big(\frac{1}{|B|}\int_B\Big(\prod_{i=1}^{2}|f_i(t_i x)|\prod_{i=1}^{2}|b_i(x)-b_i(t_ix)|\Big)^pdx\Big)^{1/p}\omega(t_1,t_2)dt_1dt_2\\
&&=I_1+I_2+I_3+I_4+I_5+I_6,
\end{eqnarray*}
where
\begin{eqnarray*}
&&I_1:=\displaystyle \int\limits_{0<t_{1},t_{2}<1}\left(\frac{1}{|B|}\int_{B}
\left(\Big(\prod_{i=1}^{2}|f_{i}(t_{i}x)|\Big)\Big(\prod_{i=1}^{2}|b_{i}(x)
-b_{i, B}|\Big)\right)^p\,dx\right)^{\frac 1p}\,\omega(\vec{t})\,d\vec{t},\\
&&I_2:=\displaystyle \int\limits_{0<t_{1},t_{2}<1}\left(\frac{1}{|B|}\int_{B}
\left(\Big(\prod_{i=1}^{2}|f_{i}(t_{i}x)|\Big)\Big(\prod_{i=1}^{2}|b_{i}(t_ix)
-b_{i, t_iB}|\Big)\right)^p\,dx\right)^{\frac 1p}\,\omega(\vec{t})\,d\vec{t},\\
&&I_3:=\displaystyle \int\limits_{0<t_{1},t_{2}<1}\left(\frac{1}{|B|}\int_{B}
\left(\Big(\prod_{i=1}^{2}|f_{i}(t_{i}x)|\Big)\Big(\prod_{i=1}^{2}|b_{i,B}
-b_{i, t_iB}|\Big)\right)^p\,dx\right)^{\frac 1p}\,\omega(\vec{t})\,d\vec{t},\\
&&I_4:=\displaystyle
\int\limits_{0<t_{1},t_{2}<1}\left(\frac{1}{|B|}\int_{B}
\left(\Big(\prod_{i=1}^{2}|f_{i}(t_{i}x)|\Big)\Big(\sum_{D(i, j)}
|b_{i}(x)-b_{i, B}||b_{j, B}-b_{j, t_{j}B}|\Big)\right)^p\,dx\right)^{\frac 1p}\,\omega(\vec{t})\,d\vec{t},\\
&&I_5:=\displaystyle
\int\limits_{0<t_{1},t_{2}<1}\left(\frac{1}{|B|}\int_{B}
\left(\Big(\prod_{i=1}^{2}|f_{i}(t_{i}x)|\Big)\Big(\sum_{D(i, j)}
|b_{i}(x)-b_{i, B}||b_{j}(t_j x)-b_{j, t_{j}B}|\Big)\right)^p\,dx\right)^{\frac 1p}\,\omega(\vec{t})\,d\vec{t},\\
&&I_6:=\displaystyle
\int\limits_{0<t_{1},t_{2}<1}\left(\frac{1}{|B|}\int_{B}
\left(\Big(\prod_{i=1}^{2}|f_{i}(t_{i}x)|\Big)\Big(\sum_{D(i, j)}
|b_{i,B}-b_{i,t_i B}||b_{j}(t_j x)-b_{j, t_{j}B}|\Big)\right)^p\,dx\right)^{\frac 1p}\,\omega(\vec{t})\,d\vec{t},
\end{eqnarray*}
and
\[
D(i, j):=\{(i, j)| (1, 2); (2, 1)\},\quad\quad ~b_{i, B}:=\frac{1}{|B|}\int_{B}b_{i},~\quad i=1, 2.
\]

Choose $p<s_{1}<\infty, p<s_{2}<\infty $ such that $1/s_1=1/p_1+1/q_1$,
$1/s_2=1/p_2+1/q_2$. Then by H\"{o}lder's inequality, we know that
\begin{eqnarray*}
\displaystyle I_{1}&\leq &\displaystyle
\int\limits_{0<t_{1},t_{2}<1}\prod_{i=1}^{2}\left(\frac{1}{|B|}\int_{B}
\left|f_{i}(t_{i}x)\right|^{p_{i}}dx\right)^{1/p_{i}}\prod_{i=1}^{2}\left(\frac{1}{|B|}\int_{B}
\left|b_{i}(x)-b_{i, B}\right|^{q_{i}}dx\right)^{1/q_{i}}\omega(\vec{t})\,d\vec{t}\\
&\leq&\displaystyle
|B|^{\lambda}\int\limits_{0<t_{1},t_{2}<1}\prod_{i=1}^{2}t_i^{n\lambda_i}\prod_{i=1}^{2}\left(\frac{1}{|t_{i}B|^{1+\lambda_i p_i}}\int_{t_{i}B}
\left|f_{i}(x)\right|^{p_{i}}dx\right)^{1/p_{i}}\\
&&\quad\displaystyle\times\prod_{i=1}^{2}\left(\frac{1}{|B|}\int_{B}
\left|b_{i}(x)-b_{i, B}\right|^{q_{i}}dx\right)^{1/q_{i}}\omega(\vec{t})\,d\vec{t}\\
&\leq&\displaystyle
 C|B|^{\lambda}\|b_1\|_{\dot{CMO}^{q_1}}\|b_2\|_{\dot{CMO}^{q_2}}
 \|f_1\|_{\dot{B}^{p_1,\lambda_1}}\|f_2\|_{\dot{B}^{p_2,\lambda_2}}
 \int\limits_{0<t_{1},t_{2}<1}\prod_{i=1}^{2}t_i^{n\lambda_i}\omega(\vec{t})\,d\vec{t}.
\end{eqnarray*}

Similarly, we  obtain
\begin{eqnarray*}
\displaystyle I_{2}&\leq &\displaystyle
\int\limits_{0<t_{1},t_{2}<1}\prod_{i=1}^{2}\left(\frac{1}{|B|}\int_{B}
\left|f_{i}(t_{i}x)\right|^{p_{i}}dx\right)^{1/p_{i}}\prod_{i=1}^{2}\left(\frac{1}{|B|}\int_{B}
\left|b_{i}(t_{i}x)-b_{i, t_{i}B}\right|^{q_{i}}dx\right)^{1/q_{i}}\omega(\vec{t})\,d\vec{t}\\
&\leq&\displaystyle
|B|^{\lambda}\int\limits_{0<t_{1},t_{2}<1}\prod_{i=1}^{2}t_i^{n\lambda_i}\prod_{i=1}^{2}\left(\frac{1}{|t_{i}B|^{1+\lambda_i p_i}}\int_{t_{i}B}
\left|f_{i}(x)\right|^{p_{i}}dx\right)^{1/p_{i}}\\
&&\quad\displaystyle\times\prod_{i=1}^{2}\left(\frac{1}{|t_iB|}\int_{t_iB}
\left|b_{i}(x)-b_{i, t_iB}\right|^{q_{i}}dx\right)^{1/q_{i}}\omega(\vec{t})\,d\vec{t}\\
&\leq&\displaystyle
 C|B|^{\lambda}\|b_1\|_{\dot{CMO}^{q_1}}\|b_2\|_{\dot{CMO}^{q_2}}\|f_1\|_{\dot{B}^{p_1,\lambda_1}}
 \|f_2\|_{\dot{B}^{p_2,\lambda_2}}
\int\limits_{0<t_{1},t_{2}<1}\prod_{i=1}^{2}t_i^{n\lambda_i}\omega(\vec{t})\,d\vec{t}.
\end{eqnarray*}

It follows from $1/p=1/s_{1}+1/{s_2}$ that $1=p/s_{1}+p/{s_2}$.
From $1/s_1=1/p_1+1/q_1,1/s_2=1/p_2+1/q_2$ and H\"{o}lder's inequality, we deduce that
\begin{eqnarray*}
\displaystyle
I_3&=&\displaystyle \int\limits_{0<t_{1},t_{2}<1}\left(\frac{1}{|B|}\int_{B}
\left(\Big(\prod_{i=1}^{2}|f_{i}(t_{i}x)|\Big)\Big(\prod_{i=1}^{2}|b_{i,B}
-b_{i, t_iB}|\Big)\right)^p\,dx\right)^{1/p}\,\omega(\vec{t})\,d\vec{t}\\
&\leq&\displaystyle\int\limits_{0<t_{1},t_{2}<1}\prod_{i=1}^{2}\left(\frac{1}{|B|}\int_B|f_{i}(t_{i}x)|^{s_i}\right)^{1/s_{i}}\left(\prod_{i=1}^{2}|b_{i, B}-b_{i, t_{i}B}|\right)\omega(\vec{t})\,d\vec{t}\\
&\leq&\displaystyle C|B|^{\lambda}\int\limits_{0<t_{1},t_{2}<1}t_1^{n\lambda_1}t_2^{n\lambda_2}\prod_{i=1}^{2}\left(\frac{1}{|t_iB|^{1+\lambda_i p_i}}\int_{t_iB}|f_{i}(t_{i}x)|^{p_i}\right)^{1/p_{i}}\\
&&\quad\times\displaystyle\left(\prod_{i=1}^{2}|b_{i, B}-b_{i, t_{i}B}|\right)\omega(\vec{t})\,d\vec{t}\\
&\leq&\displaystyle C|B|^{\lambda}\|f_1\|_{\dot{B}^{p_1,\lambda_1}}\|f_2\|_{\dot{B}^{p_2,\lambda_2}}\int\limits_{0<t_{1},t_{2}<1}t_1^{n\lambda_1}t_2^{n\lambda_2}\prod_{i=1}^{2}|b_{i, B}-b_{i, t_{i}B}|\omega(\vec{t})\,d\vec{t}\\
&\leq&\displaystyle C|B|^{\lambda}\|f_1\|_{\dot{B}^{p_1,\lambda_1}}
\|f_2\|_{\dot{B}^{p_2,\lambda_2}}\sum_{\ell=0}^\infty\sum^{\infty}_{k=0}
\int\limits_{{2^{-\ell-1}}\leq t_1<{2^{-\ell}}}\int\limits_{{2^{-k-1}}\leq t_2<{2^{-k}}}t_1^{n\lambda_1}t_2^{n\lambda_2}\\
&&\quad\displaystyle\times\left(\sum^\ell_{j=0}\left|b_{1,2^{-j}B}-b_{1,2^{-j-1}B}\right|
+\left|b_{1,2^{-k-1}B}-b_{1,t_1B}\right|\right) \\
&&\quad\times \left(
\sum^k_{j=0}\left|b_{2,2^{-j}B}-b_{2,2^{-j-1}B}\right|
+\left|b_{2,2^{-k-1}B}-b_{2,t_2B}\right|\right)\omega(\vec{t})\,d\vec{t}\\
&\leq&\displaystyle
 C|B|^{\lambda}\|b_1\|_{\dot{CMO}^{q_1}}\|b_2\|_{\dot{CMO}^{q_2}}
 \|f_1\|_{\dot{B}^{p_1,\lambda_1}}\|f_2\|_{\dot{B}^{p_2,\lambda_2}}\\
&&\quad\displaystyle\times\int\limits_{0<t_{1},t_{2}<1}\prod_{i=1}^{2}
t_i^{n\lambda_i}\omega(\vec{t})\log\frac{2}{t_1}\log\frac{2}{t_1}\,d\vec{t},
\end{eqnarray*}
where we use the fact that
\begin{eqnarray*}
|b_{1,B}-b_{1,t_1B}|&\le&\sum^k_{j=0}\left|b_{1,2^{-j}B}-b_{1,2^{-j-1}B}\right|
+\left|b_{1,2^{-k-1}B}-b_{1,t_1B}\right|\\
&\leq& C(k+1)\|b\|_{\dot{CMO}^{q_1}}\leq C\log\frac{2}{t_1}\|b\|_{\dot{CMO}^{q_1}}
\end{eqnarray*}
and \begin{eqnarray*}
|b_{2,B}-b_{2,t_2B}|\leq C\log\frac{2}{t_2}\|b\|_{\dot{CMO}^{q_2}}
\end{eqnarray*}

We now estimate $I_{4}$. Similarly, we choose $1<s<\infty$ such that $1/p=1/p_1+1/p_2+1/s$ and $1/s=1/q_1+1/q_2$. Using Minkowski's inequality and H\"{o}lder's inequality yields

\begin{eqnarray*}
 I_4&=&
\int\limits_{0<t_{1},t_{2}<1}\left(\frac{1}{|B|}\int_{B}
\left(\Big(\prod_{i=1}^{2}|f_{i}(t_{i}x)|\Big)\Big(\sum_{D(i, j)}
|b_{i}(x)-b_{i, B}||b_{j, B}-b_{j, t_{j}B}|\Big)\right)^p\,dx\right)^{1/p}\,\omega(\vec{t})\,d\vec{t}\\
&\leq& \int\limits_{0<t_{1},t_{2}<1}\left[\left(\frac{1}{|B|}\int_{B}
\left(\bigg(\prod_{i=1}^{2}|f_{i}(t_{i}x)|\bigg)
\bigg(|b_{1}(x)-b_{1, B}||b_{2, B}-b_{2, t_{2}B}|\bigg)\right)^pdx\right)^{1/p}\right.
\\
&&\quad\left.
+\left(\frac{1}{|B|}\int_{B}
\left(\bigg(\prod_{i=1}^{2}|f_{i}(t_{i}x)|\bigg)\bigg(|b_{2}(x)-b_{2, B}||b_{1, B}-b_{1, t_{1}B}|\bigg)\right)^p\,dx\right)^{1/p}\right]\,\omega(\vec{t})\,d\vec{t}\\
&\leq&
\int\limits_{0<t_{1},t_{2}<1}\prod_{i=1}^{2}\left(\frac{1}{|B|}\int_{B}
\left|f_{i}(t_{i}x)\right|^{p_{i}}dx\right)^{1/p_{i}}\Bigg\{\left(\frac{1}{|B|}\int_{B}
\left|b_{1}(x)-b_{1, B}\right|^{s}dx\right)^{1/s}\\
&&\quad\times|b_{2, B}-b_{2, t_{2}B}|+
\left(\frac{1}{|B|}\int_{B}\left|b_{2}(x)-b_{2, B}\right|^{s}dx\right)^{1/s}|b_{1, B}-b_{1, t_{1}B}|\Bigg\}\omega(\vec{t})\,d\vec{t}\\
&\leq&\displaystyle
C|B|^{\lambda}\int\limits_{0<t_{1},t_{2}<1}t_1^{n\lambda_1}
t_2^{n\lambda_2}\prod_{i=1}^{2}\left(\frac{1}{|t_i B|^{1+\lambda_i p_i}}\int_{t_iB}
\left|f_{i}(x)\right|^{p_{i}}dx\right)^{1/p_{i}}\\
&&\quad\times\Bigg\{\left(\frac{1}{|B|}\int_{B}
\left|b_{1}(x)-b_{1, B}\right|^{s}dx\right)^{1/s}|b_{2, B}-b_{2, t_{2}B}|\\
&&\quad+\left(\frac{1}{|B|}\int_{B}\left|b_{2}(x)-b_{2, B}\right|^{s}dx\right)^{1/s}|b_{1, B}-b_{1, t_{1}B}|\Bigg\}\omega(\vec{t})\,d\vec{t}\\
&\leq&
C|B|^{\lambda}\|f_1\|_{\dot{B}^{q_1,\lambda_1}}\|f_2\|_{\dot{B}^{q_2},\lambda_2}\int\limits_{0<t_{1},t_{2}<1}t_1^{n\lambda_1}t_2^{n\lambda_2}\Bigg\{\left(\frac{1}{|B|}\int_{B}
\left|b_{1}(x)-b_{1, B}\right|^{s}dx\right)^{1/s}\\
&&\quad\times|b_{2, B}-b_{2, t_{2}B}|+\left(\frac{1}{|B|}\int_{B}\left|b_{2}(x)-b_{2, B}\right|^{s}dx\right)^{1/s}|b_{1, B}-b_{1, t_{1}B}|\Bigg\}\omega(\vec{t})\,d\vec{t}.
\end{eqnarray*}

From the estimates of $I_{1}$ and $I_{3}$, we deduce that
\begin{eqnarray*}I_{4}&\leq & C|B|^{\lambda}\|f_1\|_{\dot{B}^{q_1,\lambda_1}}\|f_2\|_{\dot{B}^{q_2},\lambda_2}\|b_1\|_{\dot{CMO}^{q_1}}\|b_2\|_{\dot{CMO}^{q_2}}\\
&&\quad\times\int_{0}^{1}\int_{0}^{1}
t_1^{n\lambda_1}t_2^{n\lambda_2}\omega(t_{1}, t_{2})\left(1+\sum_{i=1}^{2}\log\frac{1}{t_{i}}\right)\,d t_{1}dt_{2}.
\end{eqnarray*}
It can be deduced from the estimates of $I_{1}$, $I_{2}$, $I_{3}$ and $I_{4}$ that
$$I_{5}\leq C|B|^{\lambda}\|f_1\|_{\dot{B}^{q_1,\lambda_1}}\|f_2\|_{\dot{B}^{q_2},\lambda_2}\|b_1\|_{\dot{CMO}^{q_1}}\|b_2\|_{\dot{CMO}^{q_2}}\int\limits_{0<t_{1},t_{2}<1}t_1^{n\lambda_1}t_2^{n\lambda_2}\omega(\vec{t})\,d\vec{t}$$
and
\begin{eqnarray*}I_{6}&\leq & C|B|^{\lambda}\|f_1\|_{\dot{B}^{q_1,\lambda_1}}\|f_2\|_{\dot{B}^{q_2},\lambda_2}\|b_1\|_{\dot{CMO}^{q_1}}\|b_2\|_{\dot{CMO}^{q_2}}\\
&&\quad\times\int_{0}^{1}\int_{0}^{1}
t_1^{n\lambda_1}t_2^{n\lambda_2}\omega(t_{1}, t_{2})\left(1+\sum_{i=1}^{2}\log\frac{1}{t_{i}}\right)\,d t_{1}dt_{2}.
\end{eqnarray*}

Combining the estimates of $I_{1}$, $I_{2}$, $I_{3}$, $I_{4}$, $I_{5}$ and $I_{6}$ gives
\begin{eqnarray*}
\displaystyle \left(\frac{1}{|B|^{1+\lambda p}}\int_{B}|\mathcal{H}_{\omega}^{\vec{b}}\vec{f}(x)|^pdx\right)^{1/p}&\leq&
C|B|^{\lambda}\|f_1\|_{\dot{B}^{q_1,\lambda_1}}\|f_2\|_{\dot{B}^{q_2},\lambda_2}\|b_1\|_{\dot{CMO}^{q_1}}\|b_2\|_{\dot{CMO}^{q_2}}\\
&&\quad\displaystyle \times\int_{0}^{1}\int_{0}^{1}
t_1^{n\lambda_1}t_2^{n\lambda_2}\omega(t_{1}, t_{2})\prod_{i=1}^{2}\log\frac{2}{t_{i}}\,d t_{1}dt_{2}.
\end{eqnarray*}
This proves (i).

Now we  prove the necessity in (ii). Assume that
$$\|\mathcal{H}_{\omega}^{\vec{b}}\|_{
\dot{B}^{p_1,\lambda_1}(\mathbb{R}^{n})\times\dot{B}^{p_2,\lambda_2}(\mathbb{R}^{n})
\rightarrow \dot{B}^{p,\lambda}(\mathbb{R}^{n})}<\infty$$
whenever $\vec b=(b_1, b_2)\in\dot{\mathrm {CMO}}^{q_1}(\rn)\times\dot{\mathrm {CMO}}^{q_2}(\rn)$.
To show $\mathbb{C}_{2}<\infty$,
it suffices to prove that $\mathbb{A}_2<\infty$, $\mathbb{B}_{2}<\infty$,
$$\begin{array}{rl}
\mathbb{D}:=\displaystyle \int\limits_{0<t_{1},t_{2}<1}
\left(\prod_{i=1}^{2}t_{i}^{n\lambda_i}\right)\omega(t_1,t_2)\log
\frac{1}{t_1}\,dt_1\,dt_2<\infty,
\end{array}$$
and
$$\begin{array}{rl}
\mathbb{E}:=\displaystyle \int\limits_{0<t_{1},t_{2}<1}
\left(\prod_{i=1}^{2}t_{i}^{n\lambda_i}\right)\omega(t_1,t_2)\log
\frac{1}{t_2}\,dt_1\,dt_2<\infty.
\end{array}$$

To prove $\mathbb{B}_2<\infty$, we set $b_{1}(x):=\log|x|\in
\mathrm{BMO}(\mathbb{R}^{n})\subset\dot{\mathrm {CMO}}^{q_1}(\rn)$,
and $b_{2}(x):=\log|x|\in
\mathrm{BMO}(\mathbb{R}^{n})\subset\dot{\mathrm {CMO}}^{q_2}(\rn)$.
Define $f_{1}:=|x|^{n\lambda_1}$ and
$f_{2}:=|x|^{n\lambda_2}$ if $x\in\rn\setminus\{0\}$, and $f_1(0)=f-2(0):=0$.
Then
$$\|f_1\|_{\dot{B}^{p_1, \lambda_1}}=\left(\frac{\omega_n}{n}\right)^{-\lambda_1}\Big(\frac{1}{1+\lambda_1 p_1}\Big)^{1/p_1},\quad \|f_2\|_{\dot{B}^{p_2, \lambda_2}}=\left(\frac{\omega_n}{n}\right)^{-\lambda_2}\Big(\frac{1}{1+\lambda_2 p_2}\Big)^{1/p_2}$$
and
$$\mathcal{H}_\omega^{\vec{b}}(\vec{f})(x)=|x|^{n\lambda_1}|x|^{n\lambda_2}\int_0^1\int_0^1t_1^{n\lambda_1}t_2^{n\lambda_2}\omega(t_1,t_2)
\log\frac{1}{t_1}\log\frac{1}{t_2}dt_1dt_2.$$
Since  $1<p<p_i<\infty$, $-1/p_i<\lambda_i<0$ and $\lambda=\lambda_1+\lambda_2~(i=1, 2)$, we see that, for all $B=B(0,R)$,
\begin{eqnarray*}
 &&\Big(\frac{1}{|B|^{1+\lambda p}}\int_B |\mathcal{H}_\omega^{\vec{b}}(\vec{f})(x)|^{p}dx\Big)^{1/p}\\
&&\quad=\Big(\frac{1}{|B|^{1+\lambda p}}\int_B |x|^{n\lambda p}dx\Big)^{1/p}\int_0^1\int_0^1t_1^{n\lambda_1}t_2^{n\lambda_2}\omega(t_1,t_2)
\log\frac{1}{t_1}\log\frac{1}{t_2}dt_1dt_2\\
&&\quad =\left(\frac{\omega_n}{n}\right)^{-\lambda}\Big(\frac{1}{1+\lambda p}\Big)^{1/p}\int_0^1\int_0^1t_1^{n\lambda_1}t_2^{n\lambda_2}\omega(t_1,t_2)
\log\frac{1}{t_1}\log\frac{1}{t_2}dt_1dt_2\\
&&\quad= \|f_1\|_{\dot{B}^{p_1,\lambda_1}}\|f_2\|_{\dot{B}^{p_2,\lambda_2}}\int_0^1\int_0^1t_1^{n\lambda_1}t_2^{n\lambda_2}\omega(t_1,t_2)
\log\frac{1}{t_1}\log\frac{1}{t_2}dt_1dt_2.
\end{eqnarray*}
Thus
$\mathbb{B}_{2}\leq\|\mathcal{H}_{\omega}^{\vec{b}}\|_{
\dot{B}^{p_1,\lambda_1}(\mathbb{R}^{n})\times\dot{B}^{p_2,\lambda_2}(\mathbb{R}^{n})
\rightarrow \dot{B}^{p,\lambda}(\mathbb{R}^{n})}<\infty.$

Since the proof for $\mathbb{D}<\infty$ is similar to that for
$\mathbb{E}<\infty$, we only show
$\mathbb{E}<\infty$. To this end, for any $r\in\mathbb N$ and $R\in(0,+\infty)$, we choose
$b_{1}(x):=\chi_{[B(0,R/2)]^c}(x)\,\sin(\pi r|x|),$ and
$b_{2}(x):=\log|x|,
$
 where $[B(0,R/2)]^c:=\rn\setminus B(0, R/2)$. Obviously, we have $\vec{b}=(b_{1},b_2)\in
\mathrm{\dot{CMO}^{q_1}}(\mathbb{R}^{n})\times\mathrm{\dot{CMO}^{q_2}}(\mathbb{R}^{n}),$ and hence,
$$\|\mathcal{H}_{\omega}^{\vec{b}}\|_{
\dot{B}^{p_1,\lambda_1}(\mathbb{R}^{n})\times\dot{B}^{p_2,\lambda_2}(\mathbb{R}^{n})
\rightarrow \dot{B}^{p,\lambda}(\mathbb{R}^{n})}<\infty.$$

Let \begin{equation}\label{f1}
f_{1}(x):=
\begin{cases}
0,&\quad |x|\leq\frac{R}{2},\\
\displaystyle|x|^{n\lambda_1},&\quad
|x|>\frac{R}{2},\end{cases}
\end{equation}
and
\begin{equation}\label{f2}
f_{2}(x):=
\begin{cases}
0,&\quad |x|\leq\frac{R}{2},\\
\displaystyle|x|^{n\lambda_2},&\quad
|x|>\frac{R}{2}.
\end{cases}
\end{equation}
 Then,  we have
$$\begin{array}{rl}
\displaystyle &\mathcal{H}_{\omega}^{\vec{b}}
\vec{f}(x)\\
&\quad=\displaystyle\int\limits_{0<t_{1},t_{2}<1}
\left(\prod_{i=1}^{2}f_{i}(t_{i}x)\right)
\left(\prod_{i=1}^{2}(b_{i}(x)-b_{i}(t_{i}x))
\right)\omega(\vec{t})\,d\vec{t}\\
&\quad=\displaystyle
|x|^{n\lambda}
\int_{{\frac{R}{2|x|}}}^{1}\int_{{\frac{R}{2|x|}}}^{1}
t_{1}^{n\lambda_1}t_{2}^{n\lambda_2}\left(b_{1}(x)-b_{1}(t_{1}x)
\right)\omega(t_{1}, t_{2})\log\frac1{t_2}\,d t_{1} dt_{2}\\
&\quad =
\displaystyle
|x|^{n\lambda}b_{1}(x)
\int_{{\frac{R}{2|x|}}}^{1}\int_{{\frac{R}{2|x|}}}^{1}
t_{1}^{n\lambda_1}t_{2}^{n\lambda_2} \omega(t_{1},
t_{2})\log\frac1{t_2}\,d t_{1} dt_{2}-\eta_{d},
\end{array}$$
whenever $R/2<|x|<R$,
$$\begin{array}{rl}
\eta_{d}&=
|x|^{n\lambda}
\displaystyle\int_{{\frac{R}{2|x|}}}^{1}\int_{{\frac{R}{2|x|}}}^{1}
t_{1}^{n\lambda_1}t_{2}^{n\lambda_2}\omega(t_{1}, t_{2})b_{1}(t_1x)
\log\frac1{t_2}\,d t_{1} dt_{2}\\
&=\displaystyle
|x|^{n\lambda}
\int_{{\frac{R}{2|x|}}}^{1}\int_{{\frac{R}{2|x|}}}^{1}
t_{1}^{n\lambda_1}t_{2}^{n\lambda_2}\omega(t_{1}, t_{2})\sin(\pi
rt_1|x|) \log\frac1{t_2}\,d t_{1}\,dt_{2}.
\end{array}$$

Since $\omega$ is integrable on $[0,1]\times[0,1]$ and
$\mathbb{B}_2<\infty$, we know that
$$\displaystyle t_{1}^{n\lambda_1}t_{2}^{n\lambda_2}\omega(t_{1},t_{2})\log\frac1{t_2}$$
is integrable on
$(\frac{1}{2},1)\times(\frac{1}{2},1)$. Then, it follows from Lemma \ref{LA} that for any $\delta>0$, there exists a positive constant
$C_{R,\delta}$ that depends on $R$ and $\delta$ such that
$$\begin{array}{rl}
\displaystyle\int_{{\frac{1}{2}}}^{1}\int_{{\frac{1}{2}}}^{1}
t_{1}^{n\lambda_1}t_{2}^{n\lambda_2}\omega(t_{1}, t_{2})\,\sin(\pi rt_1)
\log\frac1{t_2}\,d t_{1} dt_{2}<\delta/2,
\end{array}$$
for all
$r>C_{R,\delta}$. Now we choose $r>\max(1/R,1)C_{R,\delta}$.
Then for any $R/2<|x|<R$,
$r|x|>C_{R,\delta}$, and hence
$$\begin{array}{rl}
\displaystyle\int_{{\frac{1}{2}}}^{1}\int_{{\frac{1}{2}}}^{1}
t_{1}^{n\lambda_1}t_{2}^{n\lambda_2}\omega(t_{1}, t_{2})\,\sin(\pi
rt_1|x|) \log\frac1{t_2}\,d t_{1} dt_{2}<\delta/2,
\end{array}$$
which further implies that $\eta_{d}<\frac{\delta}2
|x|^{n\lambda}.$
Therefore, for any $R/2<|x|<R$,
$$\begin{array}{rl}
|\mathcal{H}_{\omega}^{\vec{b}}\vec{f}(x)|&\geq \displaystyle
|x|^{n\lambda}
\left(\int_{{\frac{R}{2|x|}}}^{1}\int_{{\frac{R}{2|x|}}}^{1}
t_{1}^{n\lambda_1}t_{2}^{n\lambda_2} \omega(t_{1},
t_{2})\log\frac1{t_2}\,d t_{1} dt_{2}-\frac{\delta}2\right).
\end{array}$$

Let $\varepsilon>0$ be small enough and choose $\delta>0$ such that
$$\begin{array}{rl}
\delta<\displaystyle \int_{{\frac{R\varepsilon}{2}}}^{1}
\int_{{\frac{R\varepsilon}{2}}}^{1}
t_{1}^{n\lambda_1}t_{2}^{n\lambda_2} \omega(t_{1}
t_{2})\log\frac1{t_2}\,d t_{1} dt_{2}.
\end{array}$$
Then, for all balls $B=B(0,R)$,
$$\begin{array}{rl}
&\displaystyle\left(\frac{1}{|B|^{1+\lambda p}}\int_B|\mathcal{H}_{\omega}^{\vec{b}}\vec{f}(x)|^pdx\right)^{1/p}\\
&\quad\geq \displaystyle \left(\frac{1}{|B|^{1+\lambda p}}\int\limits_{R/2<|x|<R}
|x|^{n\lambda p}\left(\int_{{\frac{R}{2|x|}}}^{1}
\int_{{\frac{R}{2|x|}}}^{1}
t_{1}^{n\lambda_1}t_{2}^{n\lambda_2}
\omega(t_{1}, t_{2})\log\frac1{t_2}\,d t_{1} dt_{2}
-\frac{\delta}2\right)^{p}\,dx\right)^{1/p}\\
&\quad\geq \displaystyle \left(\frac{1}{|B|^{1+\lambda p}}\int\limits_{R/2<|x|<R}
|x|^{n\lambda p}\left(\int_{{\frac{R\varepsilon}{2}}}^{1}
\int_{{\frac{R\varepsilon}{2}}}^{1}
t_{1}^{n\lambda_1}t_{2}^{n\lambda_2}\omega(t_{1},
t_{2})\log\frac1{t_2}\,d t_{1} dt_{2}-\frac{\delta}2\right)^{p}\,dx\right)^{1/p}\\
&\quad\geq \displaystyle C \left(\frac{1}{|B|^{1+\lambda p}}\int\limits_{R/2<|x|<R}
|x|^{n\lambda p}\left(\int_{{\frac{R\varepsilon}{2}}}^{1}
\int_{{\frac{R\varepsilon}{2}}}^{1}
t_{1}^{n\lambda_1}t_{2}^{n\lambda_2} \omega(t_{1},
t_{2})\log\frac1{t_2}\,d t_{1} dt_{2}\right)^{p}\,dx\right)^{1/p}\\
&\quad\geq\displaystyle C\left(\frac{\omega_n}{n}\right)^{-\lambda}\Big(\frac{1}{1+\lambda p}\Big)^{1/p}\int_{{\frac{R\varepsilon}{2}}}^{1}
\int_{{\frac{R\varepsilon}{2}}}^{1}
t_{1}^{n\lambda_1}t_{2}^{n\lambda_2}\omega(t_{1},
t_{2})\log\frac1{t_2}\,d t_{1} dt_{2}\\
&\quad=\displaystyle C
\prod_{i=1}^{2}\|f_{i}\|_{\dot{B}^{p_i,\lambda_i}(\mathbb{R}^{n})}
\int_{{\frac{R\varepsilon}{2}}}^{1}
\int_{{\frac{R\varepsilon}{2}}}^{1}
t_{1}^{n\lambda_1}t_{2}^{n\lambda_2}\omega(t_{1},
t_{2}) \log\frac1{t_2}\,d t_{1} dt_{2},
\end{array}$$
which further implies that
\begin{eqnarray*}
&&\|\mathcal{H}_{\omega}^{\vec{b}}\|_{\dot{B}^{p_1,\lambda_1}(\mathbb{R}^{n})\times
\dot{B}^{p_2,\lambda_2}(\mathbb{R}^{n}) \rightarrow
\dot{B}^{p,\lambda}(\mathbb{R}^{n})}\\
&&\quad\geq \displaystyle C
\prod_{i=1}^{2}\|f_{i}\|_{\dot{B}^{p_i,\lambda_i}(\mathbb{R}^{n})}
\int_{{\frac{R\varepsilon}{2}}}^{1}
\int_{{\frac{R\varepsilon}{2}}}^{1}
t_{1}^{n\lambda_1}t_{2}^{n\lambda_2}\omega(t_{1},
t_{2}) \log\frac1{t_2}\,d t_{1} dt_{2}.
\end{eqnarray*}
Letting $\varepsilon\to 0^+$   concludes
$\mathbb{E}<\infty.$

To show that $\mathbb{A}_2<\infty$, we let
$$b_{1}(x)=b_{2}(x):=\chi_{[B(0,R/2)]^c}(x)\,
\sin(\pi r|x|),$$
where  $ R\in(0,+\infty)$ and $r\in \mathbb{N}$, and let $f_1,\ f_2$ be as in (\ref{f1}), (\ref{f2}), respectively.
Repeating the proof for
$\mathbb{E}<\infty$, we also obtain that $\mathbb{A}_2<\infty$.
Combining all above estimates then yields $\mathbb{C}_2<\infty.$
This finishes the proof of the Theorem \ref{t3}.
\end{proof}

We remark that Theorem \ref{t3} when $m=1$ is just \cite[Theorem 3.1]{FZW}.

In particular, when $n=1$
and $$\omega(\vec{t}):=\frac{1}{\Gamma(\alpha)|(1-t_{1}, \dots, 1-t_{m})|^{m-\alpha}},$$
we know that
$$\mathcal{H}_{\omega}^{\vec{b}}(\vec{f})(x)=x^{-\alpha}I^{m}_{\alpha, \vec{b}}\vec{f}(x),\,\quad x>0,$$
where
$$I^{m}_{\alpha, \vec{b}}\vec{f}(x):=\frac{1}{\Gamma(\alpha)}
\int\limits_{0<t_{1},t_{2},...,t_{m}<x}
\frac{\left(\prod_{i=1}^{m}f_{i}(t_{i})\right)\prod_{i=1}^{m}(b_{i}(x)-b_{i}(t_{i}x))}{|(x-t_{1}, \dots, x-t_{m})|^{m-\alpha}}d\vec{t}.$$
Then, as an immediate consequence of  Theorem \ref{t3},
we have the following corollary.

\begin{corollary}
Let $0<\alpha<m$. Under the assumptions of Theorem \ref{t3}, the operator
$I^{m}_{\alpha, \vec{b}}$ maps the product of central Morrey spaces
$\dot{B}^{p_1,\lambda_1}(\mathbb{R})\times \dots \times \dot{B}^{p_m,\lambda_m}(\mathbb{R})$
to $ \dot{B}^{p,\lambda}(x^{-p\alpha}dx)$.
\end{corollary}

\section{Weighted Ces\`{a}ro operator of multilinear type and its commutator}
In this section, we focus on the corresponding results for the adjoint operators of weighted multilinear
Hardy operators.

Recall that, as the adjoint operator of the weighted
Hardy operator, the \emph{weighted Ces\`{a}ro operator $G_{\omega}$} is defined by
$$G_{\omega}f(x):=\int^1_{0}f(x/t)t^{-n}\omega(t)\,dt,\hspace{3mm}x\in \mathbb{ R}^n.$$
In particular, when $\omega\equiv 1$ and $n=1$, $G_{\omega}$ is the classical
Ces\`{a}ro operator defined as
$$\displaylines{
Gf(x):=\left\{\begin{array}{ll}
\displaystyle\int^\infty_{x}\frac{f(y)}{y}\,dy,&\quad x>0,\\
\displaystyle-\int^x_{-\infty}\frac{f(y)}{y}\,dy,&\quad
x<0.\end{array}\right.}$$
When $n=1$ and $\omega(t):=\frac{1}{\Gamma(\alpha)(\frac{1}{t}-1)^{1-\alpha}}$ with $0<\alpha<1$, the operator
$G_{\omega}f(\cdot)$ is reduced to $(\cdot)^{1-\alpha}J_{\alpha}f(\cdot)$, where $J_{\alpha}$ is a variant of Weyl integral operator defined by
$$J_{\alpha}f(x):=\frac{1}{\Gamma(\alpha)}\int_{x}^{\infty}
\frac{f(t)}{(t-x)^{1-\alpha}}\frac{dt}{t}, \quad x>0$$
Moreover, it is well known that the weighted Hardy operator $H_{\omega}$
and the weighted Ces\`{a}ro operator $G_{\omega}$ are adjoint
mutually, namely,
$$\int_{\mathbb{R}^{n}}g(x)H_{\omega}f(x)\,dx=\int_{\mathbb{R}^{n}}
f(x)G_{\omega}g(x)\,dx,\eqno(4.1)$$
for all $f\in L^p(\mathbb{R}^n)$ and $g\in L^q(\mathbb{R}^n)$ with
$1<p<\infty, 1/p+1/q=1$. We refer to \cite{X,FZW} for more details.

Let the integer $m\geq 2$, and $\omega :
[0,1]\times[0,1]^m\rightarrow [0,\infty)$ be an integrable function. Let $f_{i}$ be  measurable
complex-valued functions  on $\mathbb{R}^n$, $1\leq i\leq m$. Corresponding to the weighted multilinear Hardy operators,  we define the following \emph{weighted multilinear Ces\`{a}ro operator}:
$$\mathcal{G}_{\omega}(\vec{f})(x):=
\int\limits_{0<t_{1},t_{2},\ldots,t_{m}<1}
\left(\prod_{i=1}^{m}f_{i}(x/t_{i})(t_{i})^{-n}\right)\omega(\vec{t})\,d\vec{t},\quad x\in \mathbb{ R}^n.$$
Notice that in general
$\mathcal{H}_{\omega}$ and $\mathcal{G}_{\omega}$ do not obey the commutative rule (4.1).

We also point out that, when $n=1$ and $$\omega(\vec{t}):=\frac{1}{\Gamma(\alpha)|(\frac{1}{t_{1}}-1, \dots, \frac{1}{t_{m}}-1)|^{m-\alpha}},$$ the operator
$$\mathcal{G}_{\omega}(\vec{f})(x)=x^{m-\alpha}J^{m}_{\alpha}\vec{f}(x),\,\quad x>0,$$
where
$$J^{m}_{\alpha}\vec{f}(x):=\frac{1}{\Gamma(\alpha)}
\int\limits_{x<t_{1},t_{2},...,t_{m}<\infty}\frac{\prod_{i=1}^{m}f_{i}(x_{i})}{|(t_{1}-x, \dots, t_{m}-x)|^{m-\alpha}}\frac{d\vec{t}}{\vec{t}}.$$

Similar to the argument used in Section 2, we have the following conclusions.

\begin{theorem}\label{t4}
If $f_i \in L^{p_i}(\rn)$,  $1<p, p_i<\infty$, $i=1,\ldots, m$, and $1/p=1/p_1+\cdots+1/p_m$, then $\mathcal{G}_{\omega}$ is bounded from
$L^{p_1}(\mathbb{R}^{n})\times \dots \times L^{p_m}(\mathbb{R}^{n})$ to
$ L^p(\mathbb{R}^{n})$ if and only
if
$$\mathbb{F}:=\int\limits_{0<t_{1},t_{2},...,t_{m}<1}
\left(\prod_{i=1}^{m}t_{i}^{-n(1-1/p_{i})}\right)\omega(\vec{t})\,d\vec{t}<\infty.\eqno(4.2)$$
Moreover, $$\|\mathcal{G}_{\omega}\|_{L^{p_1}(\mathbb{R}^n)\times
\dots \times L^{p_m}(\mathbb{R}^n)\rightarrow L^{p}(\mathbb{R}^n)}=\mathbb{F}.\eqno(4.3)$$
\end{theorem}

We can also deduce from Theorem \ref{t4} that

\begin{corollary}
 Let $0<\alpha<m$. Under the assumptions of Theorem 4.1, we have $J^{m}_{\alpha}$ maps the product of weighted Lebesgue spaces
$L^{p_1}(\mathbb{R})\times \dots \times L^{p_m}(\mathbb{R})$
to $ L^p(x^{pm-p\alpha} dx)$
with norm
$$\frac{1}{\Gamma(\alpha)}\int\limits_{0<t_{1},t_{2},...,t_{m}<1}\left(\prod_{i=1}^{m}
t_{i}^{-(1-1/p_{i})}\right)\frac{1}{|(\frac{1}{t_{1}}-1, \dots, \frac{1}{t_{m}}-1)|^{m-\alpha}}\,d\vec{t}.$$
\end{corollary}

Next, we define the commutator of weighted Ces\`{a}ro operators of multilinear type as
$$\mathcal{G}_{\omega}^{\vec{b}}(\vec{f})(x):=
\int\limits_{0<t_{1},t_{2},...,t_{m}<1}\left(\prod_{i=1}^{m}f_{i}(x/t_{i})(t_{i})^{-n}\right)
\left(\prod_{i=1}^{m}\left(b_{i}(x)-b_{i}(\frac{x}{t_{i}})\right)
\right)\omega(\vec{t})\,d\vec{t},\,\quad x\in \mathbb{ R}^n.$$
In particular, we know that
$$\mathcal{G}_{\omega}^{\vec{b}}(\vec{f})(x)=x^{m-\alpha}J^{m}_{\alpha, \vec{b}}\vec{f}(x),\,\, x>0,$$
where
$$J^{m}_{\alpha, \vec{b}}\vec{f}(x):=\frac{1}{\Gamma(\alpha)}\int\limits_{x<t_{1},t_{2},...,t_{m}<\infty}
\frac{\left(\prod_{i=1}^{m}f_{i}(x_{i})\right)\prod_{i=1}^{m}
(b_{i}(x)-b_{i}(x/t_{i}))}{|(t_{1}-x, \dots, t_{m}-x)|^{m-\alpha}}\frac{d\vec{t}}{\vec{t}}.$$

 Let $m \in \mathbb{N}$ and $m\geq 2$. Define
$$\mathbb{F}_{m}:=\int\limits_{0<t_{1},t_{2}<,...,<t_{m}<1}\left(\prod_{i=1}^{m}t_{i}^{-n\lambda_i-n}\right)\omega(\vec{t})\prod_{i=1}^{m}\log\frac{2}{t_{i}}\,d\vec{t}.$$

Similar to the arguments in Section 3, we have the following conclusion.

\begin{theorem} \label{t5}
If $f_i \in L^{p_i}(\rn)$,  $1<p< p_i<\infty, 1<q_i<\infty$, $-1/p_i<\lambda_i<0$, $i=1,\ldots, m$, and $\frac{1}{p}=\frac{1}{p_1}+ \cdots
+\frac{1}{p_m}+\frac{1}{q_1}+ \cdots
+\frac{1}{q_m}$, $\lambda=\lambda_1+\cdots+\lambda_m$.

$\rm(i)$ If $\mathbb{F}_{m}<\infty$, then
$\mathcal{G}_{\omega}^{\vec{b}} $ is bounded from
$\dot{B}^{p_1, \lambda_1}(\mathbb{R}^{n})\times \cdots \times \dot{B}^{p_m, \lambda_m}(\mathbb{R}^{n})$ to $ \dot{B}^{p, \lambda}(\mathbb{R}^{n})$,
for all $\vec{b}=(b_1,b_2,\ldots,b_m)\in \dot{\mathrm{CMO}}^{q_1}(\mathbb{R}^{n})\times\cdots
\times\dot{\mathrm{CMO}}^{q_m}(\mathbb{R}^{n})$.

$\rm(ii)$ Assume that $\lambda_1p_1=\cdots=\lambda_mp_m$. In this case the condition
$\mathbb{F}_{m}<\infty$ in (i) is also necessary.
\end{theorem}

As an immediate corollary, we have the following consequence.

\begin{corollary} Let $0<\alpha<m$. Under the assumptions of Theorem \ref{t5},
we have $J^{m}_{\alpha, \vec{b}}$ maps the product of weighted Lebesgue spaces
$\dot{B}^{p_1, \lambda_1}(\mathbb{R})\times \dots \times
\dot{B}^{p_m, \lambda_m}(\mathbb{R})$ to $ \dot{B}^{p, \lambda}(x^{pm-p\alpha}dx)$
\end{corollary}

Finally, we give some further comments on weighted product Hardy operators.
Let $\omega :
[0,1]\times[0,1]\rightarrow [0,\infty)$ be an integrable function. Let $f(x_{1}, x_{2})$
be  measurable
complex-valued functions  on $\mathbb{R}^n\times\mathbb{R}^m$. The \emph{weighted product Hardy operator} is defined as
$$\mathbb{H}_{\omega}f(x_{1}, x_{2}):=
\int\limits_{0<t_{1},t_{2}<1}f(t_{1}x_{1}, t_{2}x_{2})\omega(t_{1}, t_{2})
\,dt_{1}dt_{2}, \quad (x_1,x_2)\in \rn\times \mathbb{R}^m.$$

If $\omega\equiv 1$ and $n, m=1$, then $\mathbb{H}_{\omega}f$ is
reduced to the two dimensional Hardy operator $\mathbb{H}$ defined by
$$\mathbb{H}f(x_{1}, x_{2}):=
\frac{1}{x_{1}}\frac{1}{x_{2}}\int^{x_{1}}_{0}
\int^{x_{2}}_{0}f(t_{1}, t_{1})\,dt_{1}dt_{2},\,\quad x_{1}, x_{2}\neq0,$$
which is first introduced by Sawyer \cite{S}.
The sharp estimates for weighted product Hardy
operators and their commutators on Lebesgue
spaces will be interesting questions.
\medskip

\noindent{\bf Acknowledgements.}\quad  The authors cordially thank the referees for their careful reading and helpful
comments.

\end{document}